\documentclass[titlepage,a4paper,11pt]{amsart} 
\usepackage[foot]{amsaddr}
\usepackage{amssymb}
\usepackage[lite]{amsrefs}
\usepackage{t1enc}
\usepackage[latin1]{inputenc}
\usepackage[english]{babel}
\usepackage{mathrsfs} 
\usepackage{hyperref}
\usepackage[all]{xy} 
\usepackage[lite]{amsrefs}
\usepackage[left=3cm,top=3cm,right=3cm,bottom=3cm]{geometry}

\newtheorem{thm}{Theorem}[section]
\newtheorem{prop}[thm]{Proposition}
\newtheorem{lem}[thm]{Lemma}
\newtheorem{cor}[thm]{Corollary}
\newtheorem*{defn}{Definition}
\newtheorem{rem}[thm]{Remark}
\newtheorem{ex}[thm]{Example}

\DeclareMathOperator\Le{L}
\DeclareMathOperator\Gr{Gr}

\DeclareMathOperator\Tr{Tr}

\DeclareMathOperator\spec{spec}

\DeclareMathOperator\oo{{\mathcal O}}
\DeclareMathOperator\C{\mathbb{C}}
\DeclareMathOperator\R{\mathbb{R}}
\DeclareMathOperator\N{\mathbb{N}}
\DeclareMathOperator\Z{\mathbb{Z}}

\DeclareMathOperator\radius{radius}

\DeclareMathOperator\id{\mathbf{1}}

\def\u{{\mathfrak u}}

\hypersetup{
  colorlinks   = true, 
  urlcolor     = blue, 
  linkcolor    = blue, 
  citecolor   = red    
}
\title{Circumcenters in Finsler unitary groups}
\date{\today}
\author{Martin Miglioli}
\email[Martin Miglioli]{martin.miglioli@gmail.com}
\address[Martin Miglioli]{Instituto Argentino de Matem\'atica-CONICET. Saavedra 15, Piso 3, (1083) Buenos Aires, Argentina}
\thanks{The author was supported by IAM-CONICET, grants PIP 2010-0757 (CONICET) and PICT 2010-2478 (ANPCyT)}

\begin{document}
\begin{abstract}
We study convexity properties of distance functions in Finsler unitary groups, where the Finsler structure is defined by translation of the $p$-Schatten norm on the Lie algebra. As a result we prove the existence of circumcenters for sets with radius less than $\pi/2$ in several metrics. This result is applied to a fixed point property and to quantitative metric bounds in certain rigidity problems. Bounds for convexity, existence of circumcenters and rigidity are shown to be optimal. 

\textbf{NOTE: This is an unpublished manuscript. The results in this manuscript were generalized to the infinite dimensional context in the preprint "Geometry of infinite dimensional groups: convexity and fixed points", arXiv:2203.06315, M. Miglioli}\\

\medskip

\end{abstract}

\maketitle
\tableofcontents




\section{Introduction}

In this article we address convexity properties of distance functions in Finsler unitary groups. The Finsler structures are defined by translation of the $p$-Schatten norm on the Lie algebra. See the work \cites{pschatten,fredholm,optimal} where these groups with Finsler metrics were studied in the finite and infinite dimensional context. We establish convexity results for the metrics and apply these to the existence of circumcenters. We show that the radius bounds on the sets that ensure existence are optimal. From these fixed point properties follow. 

Another approach is based on the construction of the Riemannian center of mass or Frechet-Karcher mean. The study of these means began in \cites{karcher,karcher2,karcher3} and was applied to some rigidity problems, that is, metric conditions for the equivalence of representations. They were also applied to the approximation of quasi-representations by representations, see Section 3 in \cite{karcher3}. These centers of mass are generalizations of the Frechet center of mass and are defined for probability measures.

The existence of circumcenters in our article follows from the strong convexity properties of the function 
$$v\mapsto d_2(v,u)^2$$
where $d_2$ is the distance obtained from the standard Riemannian structure. A real valued function $f$ on an interval is said to be strongly convex if $f(t)-\lambda t^2$ is convex for a $\lambda>0$, and the above defined function is strongly convex in a domain if it is a strongly convex function when it is composed with a unit-speed geodesic in the domain.

The function
$$v\mapsto \sup_{a\in A}d_p(v,a)^p$$
is only strictly convex on a domain if $A$ is finite. An optimal radius of convexity for the domains is obtained. When the radius of $A$ is less than $\pi/2$ its minimizer is the circumcenter of $A$. For infinite $A$ we obtain circumcenters with metrics which are slight perturbations of the $d_p$ metric with the $d_2$ metric to get strong convexity. Letting $p\to\infty$ we can write several radius bounds in term of the distance $d_\infty$ derived from the operator norm on the Lie algebra. We also restrict the study to geodesic subspaces of $U$ such as the special unitary and orthogonal groups, and Grassmannians embedded as spaces of symmetries.  

The article is organized as follows. In Section \ref{prel} we review some facts on the metric geometry of unitary groups endowed with Finsler structures derived from $p$-Schatten norms. In Section \ref{sprop} we establish strong convexity properties of the distance function. We first obtain the optimal radius of convexity, and we establish strong convexity properties for $2$-powers of the distance $d_2$. In Section \ref{scirc} we establish the existence of circumcenters for subsets of the unitary group with radius less than $\pi/2$ in the $d_p$ distance in the case of finite sets, and in the perturbed distances for possibly infinite sets. We apply the existence of circumcenters to prove a fixed point theorem with optimal bounds. Finally, in Section \ref{srigid} the fixed point theorem is applied to optimal metrical bounds in certain rigidity results for representations.  

\section{Preliminaries}\label{prel}

In this section we recall some geometric facts about the spaces of unitaries endowed with a bi-invariant Finsler metric derived from $p$-Schatten norms on the Lie algebra, with $2\leq p< \infty$.  Throughout this article, we will  use the metric and the geodesic structure of this space, which was studied in \cites{pschatten,fredholm,optimal}.

For $n\in \N$ let 
$$U_n=\{u\in M_n(\C):u^*u=\id\}$$
be the group of unitaries and let
$$\u_n=\{x\in M_n(\C):x^*=-x\}$$
be its Lie algebra. We will often omit the dimension $n$ from the notation. On the Lie algebra we define the following $p$-Schatten norm
$$\|x\|_p=\Tr(|x|^p)^{\frac{1}{p}}=\Tr((x^*x)^\frac{p}{2})^{\frac{1}{p}}$$
for $2\leq p<\infty$. We define $\|x\|_\infty$ as the operator norm of $x$:
$$\|x\|_\infty=\sup_{\xi\in \u}\frac{\|x\xi\|}{\|\xi\|}.$$
These norms are invariant by conjugation by unitaries, so by right or left translation we can define a norm on the tangent spaces at all points of $U$.

We introduce the Finsler metric which is Riemannian if $p = 2$. Let $\Le_p$ denote the length functional for piecewise smooth curves $\alpha$ in $U$, measured with the $p$-Schatten or operator norm:
$$\Le_p(\alpha)=\int_{t_0}^{t_1}\|\dot{\alpha}(t)\|_pdt,$$
and we define $d_p$ as the rectifiable distance in $U$
$$d_p(u_1,u_2)= \inf\{\Le_p(\gamma):\gamma \subseteq U \mbox{ joins } u_1 \mbox{ and }u_2\}.$$
This metric is invariant by left and right translations. 

The norms $\|\cdot\|_p$ for $2\leq p<\infty$ have strict convexity properties which the norm $\|\cdot\|_\infty$ does not have, but $\|x\|_p\to \|x\|_\infty$ as $p\to \infty$. Next we show that that as $p\to \infty$ the balls $B_p(u,r)$ become larger and tend to $B_\infty(u,r)$. 

\begin{prop}\label{compare}
For $u\in U$ and $2\leq p<\infty$ we have
$$B_p(u,r)\subseteq B_\infty(u,r)\mbox{  and  }B_\infty(u,r)\subseteq B_p(u,n^{1/p}r),$$
and the same inclusions hold for the closed balls.
\end{prop}

\begin{proof}
The norm $\|x\|_p$ is the $p$-norm of the eigenvalues of $x$, therefore 
$$\|x\|_\infty \leq \|x\|_p\leq n^{1/p}\|x\|_\infty$$
from which $L_\infty(\alpha) \leq L_p(\alpha)\leq n^{1/p}L_\infty(\alpha)$ follows, so that
$$d_\infty(u,v) \leq d_p(u,v)\leq n^{1/p}d_\infty(u,v)$$
for $p\geq 2$ and $u,v\in U$. Hence the conclusion follows.
\end{proof}

We focus on the existence of metric geodesics. For $2\leq p<\infty$ one-parameter unitary groups $e^{tx}\in U$ with $x\in\u$, regarded as curves of unitaries, have minimal length in the $p$-distance, as long as $\|x\|_\infty \leq\pi$. Note that the exponential map is surjective and it is a bijection between the sets
$$\u \supseteq \{z\in \u: \|z\|_\infty < \pi\}\to\{u\in U:\|\id - u\|_\infty < 2\}$$
Moreover, $\exp: \{z\in U : \|z\|_\infty \leq \pi\}\to U$ is surjective. 

We recall facts from \cite{pschatten} concerning the minimality of geodesics in infinite dimensional unitary groups endowed with Finsler distances derived from $p$-norms. This could be derived from the general theory of Riemannian manifolds for the case $p = 2$. We state them in the finite dimensional context using the following embedding of the group $U_n$ into an infinite dimensional group of unitaries given by

$$u\mapsto\left(\begin{array}{cc}
u & 0 \\
\\
0 & \id
\end{array}\right).$$

\begin{thm}[{\cite[Theorem 3.2]{pschatten}}]\label{geodesicasp}
For $2\leq p<\infty$ consider $(U,d_p)$, the following facts hold:

(1) Let $u\in U$ and $x\in U$ with $\|x\|_\infty\leq \pi$. Then the curve $\mu(t)=ue^{tx}$,
$t\in [0,1]$, is shorter than any other piecewise smooth curve in $U$ joining the same endpoints. Moreover if $\|x\|_\infty< \pi$ then $\mu$ is unique with this property.

(2) Let $u_0,u_1\in U$. Then there exists a minimal geodesic curve joining
them. If $\|u_0-u_1\|_\infty<2$, this geodesic is unique.

(3) If $u,v\in U$, then
$$\left(\sqrt{1-\frac{\pi^2}{12}}\right)d_p(u,v)\leq \|u-v\|_p\leq d_p(u, v).$$
In particular the metric space $(U,d_p)$ is complete.

\end{thm}

Note that if $d_\infty(u,v)<\pi$ then the geodesic joining $u$ and $v$ is defined by
$$\gamma_{u,v}(t)=u\exp(t\log(u^{-1}v)),$$
and if $\|x\|_p=1$ then the geodesic $\gamma(t)=ue^{tx}$ has unit speed.

Theorem \ref{geodesicasp} (1) implies that $(U,d_p)$ is a strictly interior length space, that is, a space where length minimizing curves between arbitrary points exist and its length is equal to the distance, see Definition 2.1.10 in \cite{metric}.

We define an open ball as
$$B(u,r)=\{v\in U:d_p(u,v)<r\}.$$
The closed ball is defined with with $\leq$ and denoted with $B[u,r]$. When we want to stress that the distance $d_p$ is used we write $B_p(u,r)$. When a subspace $M$ of $U$ is considered we write $B_M(u,r)$.

Note that for $r<\pi$
\begin{align*}
B_\infty[\id,r]&=\exp(\{x\in\u:\|x\|_\infty\leq r\}\\
&=\{u\in U:\spec(u)\subseteq \exp(i[-r,r])\}.
\end{align*}

\section{Convexity properties}\label{sprop}
In this section we obtain better bounds for convexity properties in the finite dimensional context based on results of \cite{pschatten}. We recall the following theorems for the convenience of the reader and include a slight generalization of the first. We state them in the finite dimensional context.

\begin{thm}[{\cite[Theorem 3.6]{pschatten}}]\label{strict}
Let $p$ be an even integer, $u$ an element of a finite dimensional unitary group $U$, and let $\beta\subseteq B(u,\pi/2)$ be a non constant geodesic. Then the function 
$$f(t)=d_p(u,\beta(t))^p$$ 
is strictly convex.
\end{thm}

\begin{proof}
Theorem 3.6 in \cite{pschatten} does not include the case when a prolongation of $\beta$ includes $u$. Note that if a prolongation of $\beta$ includes $u$ then the prolongation is of the form $\beta(t)=ue^{tx}$ and 
$$f(t)=\|x\|^pt^p.$$
Therefore $f''(t)=p(p-1)\|x\|^{p}t^{p-2}$, and $f''(t)=0$ only for $t=0$ and $p\geq 4$ even. Hence strict convexity still holds in this case.
\end{proof}

\begin{thm}[{\cite[Theorem 2.8]{fredholm}}]\label{conv}
Let $u$ be an element of a finite dimensional unitary group $U$, and let $\beta\subseteq B_\infty(u,\pi/2)$ be a geodesic. Then the function 
$$f(t)=d_\infty(u,\beta(t))$$ 
is convex.
\end{thm}

We include the following corollary of Theorem \ref{conv}. We will not use it in the article.

\begin{prop}
Let $u\in U$ and $x\in\u$ such that $\spec(ue^{tx})\in S^1 \backslash\{-1\}$ for $t\in (0,l)$. Let 
$$\theta_{max}(t)=\max\spec(-i\log(ue^{tx}))\mbox{  and  }\theta_{min}(t)=\min\spec(-i\log(ue^{tx})).$$ 
If for an open interval $I\subseteq (0,l)$ we have $\theta_{max}(t)-\theta_{min}(t)<\pi$ for $t\in I$, then in the interval $I$ the function $\theta_{max}$ is convex and the function $\theta_{min}$ is concave. The $\pi$ bound for the difference between the maximum and minimum eigenvalue is optimal.
\end{prop}

\begin{proof}
Let $I$ be an interval as in the the statement of the proposition. If $t_0\in I$ then there is a $c\in \R$ and an open interval $J\subseteq I$ with $t_0\in I$ such that $-\pi/2<\theta_{min}(t)+c\leq \theta_{max}(t)+c<\pi/2$ and $\theta_{max}(t)+c>\vert \theta_{min}(t)+c \vert$ for $t\in J$. The first inequalities imply that the geodesic $\beta(t)=e^{ic\id}ue^{tx}$ lies in $B_\infty(\id,\pi/2)$ for $t\in J$ and the second inequality implies that 
$$\theta_{max}(t)+c=d_\infty(\id,e^{ic\id}ue^{tx})$$
for $t\in J$. By Theorem \ref{conv} $\theta_{max}(t)+c$ is convex in $J$. Since convexity is a local property we conclude that $\theta_{max}$ is convex in $I$. An analogous argument implies the concavity of $\theta_{min}$ in $I$. 

To prove the optimality of the bound $\pi$ consider the following curve in $U_2$:  

$$ue^{tx}=
 \left(\begin{array}{cc}
e^{i\theta} & 0 \\
\\
0 & e^{-i\theta}
\end{array}\right)
\left(\begin{array}{cc}
\cos(t) & \sin(t) \\
\\
-\sin(t) & \cos(t)
\end{array}\right).
$$
The eigenvalues of $ue^{tx}$ are
$$\cos(\theta)\cos(t)\pm i\sqrt{1-(\cos(\theta)\cos(t))^2}.$$
Taking $\theta=\pi/2+\epsilon$ and $t\in(-\epsilon,\epsilon)$ we see that the conclusion of the proposition does not hold for small $\epsilon>0$.
\end{proof}

\subsection{Radius of convexity}\label{sradconv}

We obtain $\pi/2$ as optimal radius of convexity.

\begin{defn}
A set $M\subseteq U$ is geodesically convex if for $u,v\in M$ there is a geodesic $\gamma$ joining $u$ and $v$ such that $\gamma\subseteq M$.
\end{defn}

\begin{prop}\label{geoconv}
If $p$ is an even integer and $u\in U$, then for each $r\leq\pi/2$ the set $B_p(u,r)$ is geodesically convex, and for each $r<\pi/2$ the set $B_p[u,r]$ is geodesically convex. 
The $\pi/2$ bound for the radius of convexity is optimal.
\end{prop}

\begin{proof}
Since left translation is isometric we can assume that $u=\id$.
For $r<\pi/2$ define the set 
$$C=\{(u,v)\in B_p(\id,r)\times B_p(\id,r):\gamma_{u,v}\subseteq B_p(\id,r)\}.$$
Note that $B_p(\id,r)\times B_p(\id,r)$ is connected since $B_p(\id,r)=\exp(B_p(0,r))$ is connected. We will show that $C$ is non empty, open and closed in the space $B_p(\id,r)\times B_p(\id,r)$, and is therefore equal to this space.

If $u,v\in B_p(\id,r/2)$, then by the triangle inequality $\gamma_{u,v}\subseteq B_p(\id,r)$, so $C$ is not empty.

The set $C$ is open, otherwise there are $(u_n,v_n)_n \to (u,v)$, with $(u,v)\in C$, and $(t_n)_n\subseteq [0,1]$ such that $\gamma_{u_n,v_n}(t_n)\notin B_p(\id,r)$. We choose a convergent subsequence $t_{n_m}\to t'$, so that $\gamma_{u_{n_m},v_{n_m}}(t_{n_m})\to \gamma_{u,v}(t')\notin B_p(\id,r)$ which is a contradiction.

The set $C$ is closed, otherwise let $(u_n,v_n)_n$ be a sequence in $C$ which converges to $(u,v)\in (B_p(\id,r)\times B_p(\id,r))\cap C^c$. For all $t\in [0,1]$ we have $\gamma_{u_{n},v_{n}}(t)\to \gamma_{u,v}(t)$ and $d_p(\id,\gamma_{u_{n_m},v_{n_m}}(t))<r$ so that $d_p(\id,\gamma_{u,v}(t))\leq r$. Suppose there is a $t\in [0,1]$ such that  $d_p(\id,\gamma_{u,v}(t))= r$. Then the convexity of the map $t\mapsto d_p(\id,\gamma_{u,v}(t))^p$ is not satisfied, see Theorem \ref{strict}.

Therefore for $r<\pi/2$ the set $B_p(\id,r)$ is geodesically convex. The set $B_p(\id,\pi/2)$ is also geodesically convex, since for $u,v\in B_p(\id,\pi/2)$ we have that $u,v\in B_p(\id,r)$ for an $r<\pi/2$. 

If $u,v\in B_p[\id,r]$ for $r<\pi/2$ then $\gamma_{u,v}\subseteq B_p(\id,r')$ for all $r'$ such that $r<r'<\pi/2$. Taking $r'\to r$ we conclude that $\gamma_{u,v}\subseteq B_p[\id,r]$.

For the optimality of the bound take a ball of center $1$ and radius $\pi/2+\epsilon$ in the group $U_1\simeq S^1$. Then the geodesic $\gamma(t)=e^{ti}$ for $t\in [-3/2 \pi+\frac{1}{2}\epsilon,-\pi/2-\frac{1}{2}\epsilon]$ connects two points inside the ball, but it is not contained in the ball.   
\end{proof}

Letting $p\to \infty$ we can obtain an analogous result for the $d_\infty$ distance. 

\begin{prop}\label{geoconvunif}
If $u\in U$ then for each $r<\pi/2$ the sets $B_\infty[u,r]$ and $B_\infty(u,r)$ are geodesically convex. If $v_1,v_2\in B_\infty[u,\pi/2]$ and $d_\infty(v_1,v_2)<\pi$, then $\gamma_{v_1,v_2}\subseteq B_\infty[u,\pi/2]$.
\end{prop}

\begin{proof}
Since left translation is isometric we assume that $u=\id$. Let $r<\pi/2$ and $u,v\in B_\infty[u,r]$. For sufficiently high even integers $p$ we have $n^{1/p}r<\pi/2$ and $u,v\in B_\infty[u,r]\subseteq B_p[u,n^{1/p}r]$ by Proposition \ref{compare}. By Proposition \ref{geoconv} we get $\gamma_{u,v}\subseteq B_p[u,n^{1/p}r]$. Letting $p$ tend to infinity we have $n^{1/p}r\to r$ so that $\gamma_{u,v}\subseteq B_p[u,r]$.
 
If $u,v\in B_\infty(\id,r)$ for $r<\frac{\pi}{2}$, then $u,v\in B_\infty[\id,r']$ for $r'<r$. By the previous paragraph $\gamma_{u,v}\subseteq B[\id,r']\subseteq B[\id,r]$.

To prove the last statement assume that $u,v\in B_\infty[\id,\pi/2]$ and $d_\infty(u,v)<\pi$. We can chose sequences $u_n\to u$ and $v_n\to v$ such that $d_\infty(u_n,v_n)<\pi$, and $u_n,v_n\in B_\infty[\id,r_n]$ and $r_n<\pi/2$ with $r_n\to \pi/2$. Then by the first statement of the proposition $\gamma_{u_n,v_n}\subseteq B_\infty[\id,r_n]\subseteq B_\infty[\id,\pi/2]$. Since $\gamma_{u_1,u_2}(t)=u_1\exp(t\log(u_1^{-1}u_2))$ the geodesics depend continuously on the endpoints, and in the limit $\gamma_{u,v}\subseteq B_\infty[\id,\pi/2]$.
\end{proof}

\begin{cor}
Let $c\geq 0$ be a real number, $v$ a unitary matrix, and $x$ a skew-adjoint matrix such that $\|x\|_\infty<\pi$, $v+v^{-1}\geq c\id$ and $(ve^x)+(ve^x)^{-1}\geq c\id$. Then 
$$(ve^{tx})+(ve^{tx})^{-1}\geq c\id$$
for all $t\in [0,1]$. 
\end{cor}

\begin{proof}
This follows from Proposition \ref{geoconvunif}, the equality
$$B_\infty[\id,r]=\{u\in U:\spec(u)\subseteq \exp(i[-r,r])\}$$
and the following property of the numerical range: a matrix $a$ has numerical range in the half-space $\{b_1+b_2 i:b_1\geq 0\}$ if and only if $a+a^*\geq 0$.
\end{proof}

\begin{rem}
The proof of Proposition \ref{geoconv} and \ref{geoconvunif} also hold in the infinite dimensional context of \cite{pschatten}. To generalize Proposition \ref{geoconvunif} one takes a sequence of projections with finite dimensional range such that $p_n\to \id$ in the strong operator topology. For $x$ in the p-Schatten class $p_nxp_x$ tends to $x$ in the operator norm so that $e^{p_nxp_n}\to e^x$ in the operator norm.
\end{rem}

The second statement in Proposition \ref{geoconvunif} implies the geodesic convexity of the symmetries, that is, of the unitaries $u$ with $\spec(u)\subseteq \{1,-1\}$. 

\begin{prop}\label{geoconvref}
If $u,v$ are symmetries such that $d_\infty(u,v)<\pi$, then the unique geodesic $\gamma_{u,v}$ consists of reflections.
\end{prop}

\begin{proof}
We note that $B[i\id,\pi/2]=\{u\in U:\spec(u)\subseteq \exp(i[0,\pi])\}$ and that $B[-i\id,\pi/2]=\{u\in U:\spec(u)\subseteq \exp(i[-\pi,0])\}$. Hence 
$$B[i\id,\pi/2]\cap B[-i\id,\pi/2]=\{u\in U:\spec(u)\subseteq \{1,-1\}\}$$
is the space of symmetries. Since by Proposition \ref{geoconvunif} $\gamma_{u,v} \subseteq B[i\id,\pi/2]$ and $\gamma_{u,v}\subseteq B[-i\id,\pi/2]$ the proposition follows.
\end{proof}

\subsection{Strong convexity}\label{slambdaconv}

We prove strong convexity properties of the function $d_2(\cdot,u)$. A real valued function on an interval is strongly convex if there is a $\lambda>0$ such that one can touch the graph of the function from below by a translation of the parabola $y = \lambda x^2$. While smooth convex functions are characterized by the inequality $f''\geq 0$, for $\lambda$-convex functions this inequality turns into $f''\geq \lambda$, see Example 4.4.4. in \cite{metric}.

\begin{defn}
A function $f:[0,l]\to \R$ is $\lambda$-convex if $f(t)-\lambda t^2$ is convex for a $\lambda>0$. In this case it is called strongly convex.
\end{defn}

Note that for $p\geq 2$ the space $(U,d_p)$ is a length space.

\begin{defn}[{\cite[Definition 9.2.17]{metric}}]
Let $X$ be a length space and $\lambda >0$. A function
$f:X\to \R$ is $\lambda$-convex if for any unit-speed geodesic $\gamma$ in $X$ the function
$t\mapsto f(\gamma(t))-\lambda t^2$ is convex.
\end{defn}

The following result is an adaptation of Theorem 3.6 in \cite{pschatten} for the case $p=2$, following Remark 3.7 of this article.

\begin{thm}\label{strong}
Let $u$ be an element of a finite dimensional unitary group $U$, let $0<r<\pi/2$, and let $\beta\subseteq B_\infty [u,r]$ be a non constant geodesic. Then the function 
$$f(t)=d_2(u,\beta(t))^p$$ 
satisfies $f''\geq \lambda$ for
$$\lambda =c^2\frac{\sin(2r)}{2r} >0,$$ where $c>0$ is the speed of the geodesic in the $d_2$ metric.
\end{thm}

\begin{proof}
We assume that $u=\id$ and we use the same notation as in the proof of Theorem 3.6 in \cite{pschatten}. 
Let $v,z\in\u$ such that $\beta(s)=e^ve^{sz}$. Let $w_s=\log(e^ve^{sz})$, and $\gamma_s(t)=e^{tw_s}$. Since $d_\infty (\id,\beta(s))=\|\dot{w}_s\|_\infty<\pi/2$ the curve $\gamma_s$ is a short geodesic in the $d_2$ metric joining $\id$ and $\beta(s)$ of length $d_2 (\id,\beta(s))=\|\dot{w}_s\|_2$. Then $f(s)=\|\dot{w}_s\|_2^2=-\Tr(w_s^2)$, hence
$$f'(s)=-2\Tr(w_s\dot{w}_s)=H_{w_s}(\dot{w}_s,w_s).$$
We denote with $H_a(\cdot,\cdot)$ the Hessian at $a\in\u$ of the function $\|\cdot\|^2_2$, and $Q_a(\cdot)$ its associated quadratic form. 

Using the formula for the differential of the exponential we get $e^{-w_s}d\exp_{w_s}(\dot{w}_s)=z$, namely 
$$z=\int_0^1 e^{-tw_s}\dot{w}_s e^{tw_s} dt.$$
Thus
$$\Tr(w_s\dot{w}_s)=\int_0^1 \Tr(w_s e^{-tw_s}\dot{w}_s e^{tw_s})dt=\Tr(zw_s).$$
Hence
$$f''(s)=-2\Tr(\dot{w}_s z)=H_{w_s}(\dot{w}_s,z),$$
and if we put $\delta_s(t)= e^{-tw_s}\dot{w}_s e^{tw_s}$, then
$$f''(s)=\int_0^1 -2\Tr(\delta_s(0)\delta_s(t))dt=\int_0^1 H_{w_s}(\delta_s(0)\delta_s(t))dt$$
We define $R_s^2=Q_{w_s}(\dot{w}_s)=\|\dot{w}_s\|_2^2>0$ and note that $\delta_s$ lies in the sphere of radius $R_s$ of the Hilbert space $\u$ endowed with the inner product $H_{w_s}$, hence
$$H_{w_s}(\delta_s(0),\delta_s(t))=R_s^2\cos(\alpha_s(t)),$$
where $\alpha_s(t)$ is the angle subtended by $\delta_s(0)$ and $\delta_s(t)$. If $L_0^t(\delta_s)$ is the length in the sphere of the curve $\delta_s$ from $\delta_s(0)$ to $\delta_s(t)$ then
\begin{align*}
R_s\alpha_s(t)&\leq L_0^t(\delta_s)=\int_0^tQ_{w_s}^\frac{1}{2}( e^{-tw_s}[w_s,\dot{w}_s] e^{tw_s})dt)\\
&=\int_0^tQ_{w_s}^\frac{1}{2}([w_s,\dot{w}_s])dt=tQ_{w_s}^\frac{1}{2}([w_s,\dot{w}_s]).
\end{align*}
By Property 1. of Remark 3.5 in \cite{pschatten} which states that $Q_a([b,a])\leq 4\|a\|_\infty ^2Q_a(b)$ we see that $Q_{w_s}([w_s,\dot{w}_s])\leq 4\|w_s\|_\infty^2R_s^2$. Hence
$$R_s\alpha_s(t)\leq R_s2t\|w_s\|_\infty<R_s\pi.$$
So
$$\cos(\alpha_s(t))\geq\cos(2t\|w_s\|_\infty)$$
and integrating with respect to the $t$-variable we get
$$f''(s)\geq R_s^2\frac{\sin(2\|w_s\|_\infty)}{2\|w_s\|_\infty}\geq\|z\|_2\frac{\sin(2r)}{2r}>0,$$
since from the exponential metric decreasing property in Lemma 3.3 of \cite{pschatten} $R_s^2=\|\dot{w}_s\|_2^2\geq \|z\|_2^2$, where $\|z\|_2$ is the speed of the geodesic.
\end{proof}

\section{Existence of circumcenters}\label{scirc}

In this section we prove the existence of circumcenters for subsets of radius less than $\pi/2$ in the $d_p$ metric for finite subsets and in perturbed $d_p$ metrics in general. We show that this radius is optimal. We prove the results in geodesic subspaces. The existence result is based on the following existence result of minimizers for $\lambda$-convex functions 

\begin{prop}[{\cite[Proposition 9.2.20]{metric}}]\label{min}
Let $X$ be a complete space with a strictly intrinsic
metric, and $f:X\to \R$ a continuous $\lambda$-convex function for a $\lambda>0$ which is bounded from below. Then $f$ has a unique minimum point.
\end{prop}

Complete geodesic spaces with $\lambda$-convex squares of distance functions have well defined circumcenters for its bounded subsets, see Proposition 9.2.24 in \cite{metric}. The main results of the article are stated in terms of geodesic subsets, the following lemma is the basic technical lemma that ensures that this can be done.

\begin{defn}\label{defgeod}
We call $M\subseteq U$ a geodesic subset if all points of $M$ are connected by a piecewise smooth path, and if for all $u,v\in M$ such that $d_\infty(u,v)<\pi$ the unique geodesic $\gamma_{u,v}$ in $U$ which joins $u$ and $v$ is contained in $M$.
\end{defn}

\begin{lem}\label{metricgeosub}
Let $M$ be a geodesic subset of $(U,d_p)$ for $p\geq 2$. Then $(M,d_p)$ has a well defined length structure, and for $r<\pi/2$ and $u$ in $M$ 
$$B_M(u,r)=B_U(u,r)\cap M.$$
The interior metric and geodesics of $M$ and $U$ agree on this set.
\end{lem}

\begin{proof}
Let $d_M$ and $d_U$ be the metrics of $M$ and $U$ respectively. Since $M\subseteq U$ we have $d_U\leq d_M$. Therefore the inclusion $\subseteq$ is straightforward. To prove the inclusion $\supseteq$ we note that if $v\in B_U(u,r)\cap M$ then Definition \ref{defgeod} asserts that the geodesic in $U$ joining $u$ and $v$ lies in $M$, so $d_M(u,v)<r$.

To prove that the two metrics agree on this set note that if $u_1,u_2$ in $B_M(u,r)$ then $d(u_1,u_2)<\pi$ so that the geodesic in $U$ joining $u_1$ and $u_2$ is contained in $M$ by definition. We conclude that $d_M(u_1,u_2)=d_U(u_1,u_2)$.     
\end{proof}


The special unitary and orthogonal groups, and Grassmannians are examples of geodesic subsets of $U$.

\begin{ex}\label{exgrass}
For $0<m<n$ we consider the Grassmannian $\Gr_{m,n}$ of $m$ dimensional subspaces of the space $\C^n$. For each subspace we consider de projection $P$ onto this subspace. We identify the space of projections with a space of symmetries via the map
$$P\mapsto 2P-\id$$
and we denote $e_P=2P-\id$. Each $e_P$ is a unitary operator so that an injection 
$$\Gr_{m,n}\to U$$
is defined. The Grassmannian $\Gr_{m,n}$ is a geodesic subspace of $U$ since if $d_p(e_P,e_Q)<\pi$ then $d_\infty(e_P,e_Q)<\pi$, so there is a geodesic $\gamma_{e_P,e_Q}$ in $U$ connecting $e_P$ and $e_Q$. By Proposition \ref{geoconvref} this geodesic lies in the space of symmetries. See also \cite{portarecht}.  
\end{ex}

\begin{ex}
The special unitary group $SU$ is a geodesic subspace of $U$. Let $u\in SU$ be such that $d_p(\id,u)<\pi/2$ in $U$. Then there is a geodesic $\gamma:[0,1]\to U$ given by $e^{tx}$ which joins $\id$ and $u=e^x$. Note that $\|x\|_p=d_p(\id,u)<\pi$ so that $\|x\|_\infty<\pi$. Also
$$e^{\Tr(x)}=\det(e^x)=1$$
so that $\Tr(x)\in 2\pi i\Z$. Since $\|x\|_\infty<\pi$ we see that $\Tr(x)=0$ which implies  $\gamma\subseteq SU$.
The exponential map is surjective since $SU$ is compact and connected.
\end{ex}

\begin{ex}
The special orthogonal group $SO$ is a geodesic subspace of $U$. It consists of the elements of $SU$ with real matrix entries, that is, if $c$ is the complex conjugation operator then $u\in SO$ if and only if $u\in SU$ and $cuc=u$. Let $u\in SO$ be such that $d_p(\id,u)<\pi/2$ in $U$. Then there is a geodesic $\gamma:[0,1]\to SU$ given by $e^{tx}$ which joins $\id$ and $u=e^x$. Note that $\|x\|_\infty\leq\|x\|_p<\pi$. Since $ce^xc=e^{cxc}=e^x$ and $\|cxc\|_\infty=\|x\|_\infty<\pi$ we see that $cxc$ and $x$ are in the domain of injectivity of the exponential so that $x=cxc$. Therefore $x$ has real entries and $\gamma\subseteq SO$.
The exponential map is surjective since $SO$ is compact and connected.
\end{ex}

\begin{defn}
Given a metric space $(M,d)$ and a subset $A\subseteq M$. A closed ball of minimal radius among the balls containing $A$ is called a circumscribed ball of $A$, its center a circumcenter of $A$. The radius of these closed balls is called the circumradius of $A$, which is
$$\radius(A) = \inf\{r:\mbox{ there is } c\in M \mbox{ such that } A\subseteq B[c,r]\}.$$
\end{defn}

\begin{defn}
We define the function $f^A:M\to \R$ by
$$f^A(u)=\sup_{a\in A}d(u,a).$$
\end{defn}

The proof of the next lemma is straightforward.

\begin{lem}
We have $\radius(A)=\inf_{u\in M}f^A(u)$ and the minimizers of $f^A$ are the circumcenters of $A$. 
\end{lem}

Hence if $f^A$ has a unique minimizer then $A$ has a unique circumcenter, which we call the circumcenter of $A$. When the radius and the function $f^A$ are defined with certain metrics we include the subscript of the metric, for example, $f^A_p$ and $\radius_p$ are defined with the $d_p$ metric.

\begin{defn}
We define the metric $d_{\infty,\epsilon}$ as
$$d_{\infty,\epsilon}=(d_\infty^2+(\epsilon d_2)^2)^\frac{1}{2},$$
and we define the metric $d_{p,\epsilon}$ as
$$d_{p,\epsilon}=(d_p^p+(\epsilon d_2^\frac{2}{p})^p)^\frac{1}{p}=(d_p^p+\epsilon^p d_2^2)^\frac{1}{p}.$$
\end{defn}

The metric $d_{\infty,\epsilon}$ is an interior metric since it is the $l^2$ norm of interior metrics. Note that 
$$d_\infty\leq d_{\infty,\epsilon}\leq d_\infty+\epsilon d_2.$$
The bound $d_2\leq n^\frac{1}{2}\pi/2$ implies $d_{\infty,\epsilon}\leq d_\infty+\epsilon n^\frac{1}{2}\pi/2$.
Since $\epsilon d_2^\frac{2}{p}$ is a metric for $p\geq 2$ it follows that $d_{p,\epsilon}$ is a metric. It satisfies
$$d_p\leq d_{p,\epsilon}\leq d_p+\epsilon d_2^\frac{2}{p}.$$
The bound $d_2\leq n^\frac{1}{2}\pi/2$ implies $d_{p,\epsilon}\leq d_p+\epsilon (n^\frac{1}{2}\pi/2)^\frac{2}{p}$. Hence $d_{p,\epsilon}\to d_p$ as $\epsilon \to 0$. Note that $d_{p,\epsilon}$ does note tend to $d_{\infty,\epsilon}$ as $p\to\infty$.

\begin{thm}\label{thmcirc}
Given $p$ an even integer, a closed geodesic subspace $M$ of $(U,d_p)$ and a finite subset $A\subseteq M$, if $\radius_p(A)<\pi/2$ then $A$ has unique circumcenter in $M$. The $\pi/2$ bound is best optimal.
\end{thm}

\begin{proof}
We show that the function $f^A_p:M\to\R$ has a unique minimizer. 

By Lemma \ref{metricgeosub} we do not distinguish between the sets $M$ and $U$ since all the arguments take place in balls of radius less that $\pi/2$, where the metrics of $M$ and $U$ agree and have the same geodesics.

Take an $r$ such that $\radius_p(A)<r<\pi/2$ and define the set
$$C=\{c\in M: A\subseteq B[c,r]\}.$$
It is easy to verify that 
$$C=\bigcap_{a\in A}B[a,r]$$
so that  $C=(f^A_p)^{-1}((-\infty,r])$ for the function $f^A(u)=\sup_{a\in A}d_p(u,a)$. For each $a\in A$ the set $B_p[a,r]$ is geodesically convex by Proposition \ref{geoconv}, hence the intersection $C$ is geodesically convex. It is also closed since it is the intersection of closed balls. 

Since the function $(f^A_p)^p=\sup_{a\in A}d(u,a)^p$ is bounded from below and it is a supremum of a finite number of continuous and strictly convex functions it is continuous and strictly convex, therefore it has a unique minimizer in $C$, and hence also in $M$.  

We show that the $\pi/2$ bound cannot be improved. Take $M=U_1\simeq S^1$ and $A=\{1,-1\}$. Then $\radius_p(A)=\pi/2$ but $f_A$ has two minimizers $\{i,-i\}$.

\end{proof}

\begin{rem}
We plan to generalize the previous theorem to the infinite dimensional context of \cite{pschatten}.
\end{rem}

\begin{lem}\label{convie}
For $r<\pi/2$ and $u\in U$ the balls $B_{\infty,\epsilon}[u,r]$ are geodesically convex.
\end{lem}

\begin{proof}
We assume that $u=\id$. Since $d_\infty\leq d_{\infty,\epsilon}$ we have
$$B_{\infty,\epsilon}\subseteq B_{\infty}[\id,r].$$
Therefore, if $u,v\in B_{\infty,\epsilon}[\id,r]$, then $u,v\in B_{\infty}[\id,r]$ and by Proposition \ref{geoconvunif} there is a unique geodesic $\gamma_{u,v}$ joining $u$ and $v$ which is contained in $B_{\infty}[\id,r]$. By Theorem \ref{conv} the function $d_p(\gamma_{u,v}(t),\id)^p$ is convex, and by Theorem \ref{strong} the function $d_2(\gamma_{u,v}(t),\id)^2$ is strongly convex. Hence the function
$$d_{\infty,\epsilon}(\gamma_{u,v}(t),\id)^2=d_\infty(\gamma_{u,v}(t),\id)^2+\epsilon^2 d_2(\gamma_{u,v}(t),\id)^2$$
is convex and we conclude that $\gamma_{u,v}\subseteq B_{p,\epsilon}[\id,r]$.   

\end{proof}

\begin{thm}\label{thmcirc3}
Given $p$ an even integer, an $\epsilon>0$, a closed geodesic subspace $M$ of $U$ and a subset $A\subseteq M$. If $\radius_{\infty,\epsilon}(A)<\pi/2$ then the set $A$ has a unique circumcenter in $(M,d_{\infty,\epsilon})$.
\end{thm}

\begin{proof}
By Lemma \ref{metricgeosub} we do not distinguish between the sets $M$ and $U$ since all the arguments take place in balls of radius less that $\pi/2$, where the metrics of $M$ and $U$ agree and have the same geodesics.

Take an $r$ such that $\radius_{\infty,\epsilon}(A)<r<\pi/2$ and define the set
$$C=\{c\in M: A\subseteq B_{\infty,\epsilon}[c,r]\}.$$
It is easy to verify that 
$$C=\bigcap_{a\in A}B_{\infty,\epsilon}[a,r]$$
so that $C=(f^A_{\infty,\epsilon})^{-1}((-\infty,r])$ for the function $f^A_{\infty,\epsilon}(u)=\sup_{a\in A}d_{\infty,\epsilon}(u,a)$. For each $a\in A$ the set $B_{\infty,\epsilon}[a,r]$ is geodesically convex by Lemma \ref{convie} hence the intersection $C$ is geodesically convex. It is also closed since it is the intersection of closed balls. 

Note that
$$f^A_{\infty,\epsilon}(u)^2=\sup_{a\in A}d_{\infty,\epsilon}(u,a)^2=\sup_{a\in A}(d_\infty(u,a)^2+\epsilon^2 d_2(u,a)^2).$$

Let $\beta:[0,l]\to C$ be a unit speed geodesic in the $d_2$ metric given by $\beta(t)=ue^{tx}$. Therefore $f(t)=d_2(\beta(t),a)^2$ satisfies 
$$f''\geq \frac{\sin(2r)}{2r}=\lambda>0$$ 
for all $a\in A$. Since $g(t)=d_\infty(\beta(t),a)^2$ is convex, then 
$$d_\infty(\cdot,a)^2+\epsilon^2 d_2(\cdot,a)^2$$ 
is $\epsilon^2\lambda$-convex in $(C,d_2)$. Whence $(f^A_{\infty,\epsilon})^2$ is $\epsilon^2\lambda$-convex in $(C,d_2)$ since it is the supremum of $\epsilon^2\lambda$-convex functions.     

Since the function $(f^A_{\infty,\epsilon})^2$ is bounded from below and it is $\epsilon^2\lambda$-convex when it is composed with any unit speed geodesic of $(C,d_2)$ we conclude from Proposition \ref{min} that it has a unique minimizer in $C$, and hence also in $M$. This unique minimizer is the the circumcenter of $A$ in the $d_{\infty,\epsilon}$ metric. 

\end{proof}

\begin{lem}\label{convpe}
For $r<\pi/2$ and $u\in U$ the balls $B_{p,\epsilon}[u,r]$ are geodesically convex.
\end{lem}

\begin{proof}
We assume that $u=\id$. Since $d_\infty\leq d_p\leq d_{p,\epsilon}$ we have
$$B_{p,\epsilon}[\id,r]\subseteq B_{p}[\id,r]\subseteq B_{\infty}[\id,r].$$
Therefore, if $u,v\in B_{p,\epsilon}[\id,r]$, then $u,v\in B_{\infty}[\id,r]$ and by Proposition \ref{geoconvunif} there is a unique geodesic $\gamma_{u,v}$ joining $u$ and $v$ which is contained in $B_{p}[\id,r]\subseteq B_{\infty}[\id,r]$. By Theorem \ref{strict} the function $d_p(\gamma_{u,v}(t),\id)^p$ is strictly convex, and by Theorem \ref{strong} the function $d_2(\gamma_{u,v}(t),\id)^2$ is strongly convex. Hence the function
$$d_{p,\epsilon}(\gamma_{u,v}(t),\id)^p=d_p(\gamma_{u,v}(t),\id)^p+\epsilon^p d_2(\gamma_{u,v}(t),\id)^2$$
is convex and we conclude that $\gamma_{u,v}\subseteq B_{p,\epsilon}[\id,r]$.   

\end{proof}

\begin{thm}\label{thmcirc2}
Given $p$ an even integer, an $\epsilon>0$, a closed geodesic subspace $M$ of $U$ and a subset $A\subseteq M$. If $\radius_{p,\epsilon}(A)<\pi/2$ then the set $A$ has a unique circumcenter in $(M,d_{p,\epsilon})$.
\end{thm}

\begin{proof}
By Lemma \ref{metricgeosub} we do not distinguish between the sets $M$ and $U$ since all the arguments take place in balls of radius less that $\pi/2$, where the metrics of $M$ and $U$ agree and have the same geodesics.

Take an $r$ such that $\radius_{p,\epsilon}(A)<r<\pi/2$ and define the set
$$C=\{c\in M: A\subseteq B_{p,\epsilon}[c,r]\}.$$
It is easy to verify that 
$$C=\bigcap_{a\in A}B_{p,\epsilon}[a,r]$$
so that $C=(f^A_{p,\epsilon})^{-1}((-\infty,r])$ for the function $f^A_{p,\epsilon}(u)=\sup_{a\in A}d_{p,\epsilon}(u,a)$. For each $a\in A$ the set $B_{p,\epsilon}[a,r]$ is geodesically convex by Lemma \ref{convpe} hence the intersection $C$ is geodesically convex. It is also closed since it is the intersection of closed balls. 

Note that
$$f^A_{p,\epsilon}(u)^p=\sup_{a\in A}d_{p,\epsilon}(u,a)^p=\sup_{a\in A}(d_p(u,a)^p+\epsilon^p d_2(u,a)^2).$$

Let $\beta:[0,l]\to C$ be a unit speed geodesic given by $\beta(t)=ue^{tx}$. The unit speed condition means that $\|x\|_p=1$ and this geodesic has speed $\|x\|_2$ in the $d_2$ metric. Note that $1=\|x\|_p\leq\|x\|_2$. Therefore $f(t)=d_2(\beta(t),a)^2$ satisfies 
$$f''\geq \frac{\sin(2r)}{2r}=\lambda>0$$ 
for all $a\in A$. Since $g(t)=d_p(\beta(t),a)^p$ satisfies $g''\geq 0$ then 
$$d_p(\cdot,a)^p+\epsilon^p d_2(\cdot,a)^2$$ 
is $\epsilon^p\lambda$-convex in $(C,d_p)$. Whence $(f^A_{p,\epsilon})^p$ is $\epsilon^p\lambda$-convex in $(C,d_p)$ since it is the supremum of $\epsilon^p\lambda$-convex functions.     

Since the function $(f^A_{p,\epsilon})^p$ is bounded from below and it is $\epsilon^p\lambda$-convex when it is composed with any unit speed geodesic of $(C,d_p)$ we conclude from Proposition \ref{min} that it has a unique minimizer in $C$, and hence also in $M$. This unique minimizer is the the circumcenter of $A$ in the $d_{p,\epsilon}$ metric. 

\end{proof}

We can apply this result to the following fixed point property.

\begin{thm}\label{fixedpoint}
Let $p$ be an even integer and let a group $G$ act isometrically on a closed geodesic subspace $M$ of $U$. If there is an $m\in M$ such that the radius of the orbit of $m$ with the $d_\infty$ distance is strictly less than $\pi/2$ then  the action has a fixed point. 
\end{thm}

\begin{proof}
We denote by $\oo(m)$ the orbit of $m$.  
Since $\radius_\infty(\oo(m))<\pi/2$ we can choose an $\epsilon>0$ small enough such that $\radius_{\infty,\epsilon}(\oo(m))<\pi/2$. We consider the function $f^{\oo(m)}_{\infty,\epsilon}:M\to \R$. The action is isometric for the metric $d_{\infty,\epsilon}$ and the orbit is invariant, hence
$$f^{\oo(m)}_{\infty,\epsilon}(m')=f^{g\cdot \oo(m)}_{\infty,\epsilon}(g\cdot m')=f^{\oo(m)}_{\infty,\epsilon}(g\cdot m')$$ 
for all $g\in G$ and $m'\in M$. By Theorem \ref{thmcirc3} the function $f^{\oo(m)}_{\infty,\epsilon}$ has a unique minimizer $c$ and since $f^{\oo(m)}_{\infty,\epsilon}(c)=f^{\oo(m)}_{\infty,\epsilon}(g\cdot c)$ for all $g\in G$ we conclude that $g\cdot c=c$ for all $g\in G$.  
\end{proof}

\section{Optimal bounds for rigidity}\label{srigid}

Rigidity problems ask under what conditions on two group homomorphisms $\phi,\rho:H\to G$ there is a $g\in G$ such that $\phi(h)=g\rho(h)g^{-1}$. In Proposition 4.4 1) and Lemma 2.6 of \cite{ulam} local rigidity results are obtained when $G$ is a group of unitaries. Local rigidity results assert that if $\phi$ and $\rho$ are close in some sense, so that $\{\phi(h)\rho(h)^{-1}:h\in H\}$ is small set, then a $g$ giving an equivalence between the two representations exists. 

If we take as $H$ the two element group, $G=U_1\simeq S^1$ equal to the unitary group on one dimensional space, $\phi$ the trivial representation and $\rho$ the non trivial representation, then $\{\phi(h)\rho(h)^{-1}:h\in H\}=\{-1,1\}\subseteq S^1$. Note that $\radius_\infty(\{-1,1\})=\pi/2$. We show that this radius is the smallest radius that the set $\{\phi(h)u\rho(h)^{-1}:h\in H\}$ can have in the $d_\infty$ metric for two non equivalent representations and for all $g\in G$.

\begin{rem}
We note that in Theorem 3.7 of \cite{karcher3} rigidity results are obtained using the center of mass of pushforward of Haar probability measures. In this article admissible norms are used to define bi-invariant Finsler metrics. In this context the scaled uniform norm $2\|\cdot\|_\infty$ is admissible, that is, it satisfies 
$$2\|[x,y]\|_\infty\leq 2\|x\|_\infty 2\|y\|_\infty$$ 
for all $x,y$ in the Lie algebra of the group. So the upper bounds in terms of this norm must be divided by $2$ to get bounds in terms of the not scaled norm. The bound obtained in \cite{karcher3} in terms of the $d_\infty$ distance for the equivalence of homomorphisms is $\pi/4$.
\end{rem}

\begin{thm}\label{rigrep}
Let $\phi,\rho:H\to G$ be two homomorphisms into a closed geodesic subgroup $G \subseteq U$. If there is an $u\in G$ such that $\radius_\infty(\{\phi(h)u\rho(h)^{-1}:h\in H\})<\pi/2$ then there is $g\in G$ such that 
$$\phi(h)=g\rho(h)g^{-1}$$
for all $h\in H$.
\end{thm} 

\begin{proof}
The action of $H$ on $G$ given by
$$h\mapsto (g\mapsto \phi(h)g\rho(h)^{-1})$$
is isometric in the $d_{\infty,\epsilon}$ metric by bi-invariance of the metric $d_\infty$ and $d_2$. 
We have $\oo(u)=\{\phi(h)u\rho(h)^{-1}:h\in H\}$ and $\radius_\infty(\oo(u))<\pi/2$. By Theorem \ref{fixedpoint} the action has a fixed point $c$, that is, $\phi(h)c\rho(h)^{-1}=c$ for all $h\in H$. Hence $\phi(h)=c\rho(h) c^{-1}$ for all $h\in H$. 
\end{proof}

\begin{rem}
Using the bi-invariance of the distance $d_\infty$ it is easy to verify that the condition $\radius_\infty(\{\phi(h)u\rho(h)^{-1}:h\in H\})<\pi/2$ is equivalent to the condition
$$\inf_{u,v\in G}\sup_{h\in H}d_\infty(\rho(h),v\pi(h) u)<\pi/2.$$
This is implied by the condition $\sup_{h\in H}d_\infty(\rho(h),\pi(h))<\pi/2$.

\end{rem}

For actions on Grassmannians we can obtain a similar result. Consider for integers $0<m<n$ a Grassmannian $\Gr_{m,n}\hookrightarrow U$ embedded in the unitary group as in Example \ref{exgrass}.

\begin{thm}
Let $H\subseteq U$ be a subgroup which acts in the canonical way on $Gr_{m,n}$, that is
$$h\mapsto (e_P\mapsto he_Ph^{-1}).$$
If there is a projection $P$ onto an $m$-dimensional subspace such that 
$$\radius_\infty(\{he_Ph^{-1}:h\in H\})<\pi/2$$ 
then there is a projection $Q$ onto an $m$-dimensional subspace such that $hQ=Qh$ for all $h\in H$. The $\pi/2$ bound is optimal.
\end{thm}

\begin{proof}
Since the action is isometric for the $d_{\infty,\epsilon}$ metrics and the radius in the $d_\infty$ metric of $\oo(e_P)=\{he_Ph^{-1}:h\in H\}$ is less than $\pi/2$ the action has fixed point $e_Q$ in $\Gr_{m,n}$. Therefore the projection $Q$ satisfies 
$$he_Qh^{-1}=h(2Q-\id)h^{-1}=2Q-\id,$$
for all $h\in H$, and the conclusion of the theorem follows.

To prove the optimality of the bound consider $\Gr_{1,2}$ and take the projections $P$ and $Q$ onto the first and second coordinates of $\C^2$. Then

$e_P=
 \left(\begin{array}{cc}
1 & 0 \\
\\
0 & -1
\end{array}\right)
$
,\quad
$
e_Q=
 \left(\begin{array}{cc}
-1 & 0 \\
\\
0 & 1
\end{array}\right)
$
\quad and define \quad
$
x=
 \left(\begin{array}{cc}
0 & -1 \\
\\
1 & 0
\end{array}\right).
$

Then $e^{tx}$ is a group of rotation matrices and 
$$\gamma(t)=e_Pe^{tx}=e^{-\frac{1}{2} tx}e_Pe^{\frac{1}{2}tx}$$ 
for $t\in [0,\pi]$ is a curve from $e_P$ to $e_Q$ of speed $\|x\|_\infty =1$ in the $d_\infty$ metric and total length $\pi$. Hence $d_\infty(e_P,e_Q)\leq \pi$. Take as $H$ the group consisting of maps $(x_1,x_2)\mapsto (s_1x_{d(1)},s_2x_{d(2)})$ for signs $s_1,s_2$ and permutation $d$ of $\{1,2\}$. This group has no fixed points on $\Gr_{1,2}$ and the orbit of $e_P$ is $\{e_P,e_Q\}$, hence $\radius_\infty(\{e_P,e_Q\})\geq \pi/2$. Since $d_\infty(e_P,e_Q)\leq \pi$ we conclude that $\radius_\infty(\{e_P,e_Q\})\leq \pi/2$, hence $\radius_\infty(\{e_P,e_Q\})= \pi/2$.
\end{proof}






\noindent
\end{document}